\documentclass[11pt,a4paper]{article}

\usepackage[T1]{fontenc}
\usepackage{lmodern}
\usepackage{amsmath,amssymb,amsthm,amscd,mathtools}
\usepackage{booktabs}
\usepackage[margin=1.15in]{geometry}
\usepackage{microtype}
\usepackage[colorlinks=true,citecolor=blue,linkcolor=blue,urlcolor=blue]{hyperref}

\newcommand{\PP}{\mathbb P}
\newcommand{\HH}{\mathbb H}
\newcommand{\ZZ}{\mathbb Z}

\newcommand{\Aut}{\operatorname{Aut}}
\newcommand{\Sing}{\operatorname{Sing}}

\newtheorem{theorem}{Theorem}[section]

\newtheorem{remark}[theorem]{Remark}

\title{Explicit Bicanonical Models of Eight Fake Quadrics}
\author{Lev Borisov and Carlos Rito}
\date{}

\begin{document}
\maketitle

\begin{abstract}
We compute explicit defining equations for eight fake quadrics arising as
$\mathbb Z/2\times\mathbb Z/4$-covers of two singular
$\mathbb Z/2$-Godeaux surfaces obtained in earlier work of the second
author. Starting from explicit equations for the universal covers of
the Godeaux surfaces, we reconstruct the relevant character eigenspaces
and determine the homogeneous ideals of the bicanonical models of the
eight fake quadrics in $\mathbb P^8$. All eight models are defined over
$\mathbb Q$. We prove that the surfaces are pairwise non-isomorphic and
rigid. Combined with the non-product result established in the earlier
work, this gives the first explicit projective models of fake quadrics
which are not isogenous to a product of curves.
\end{abstract}

\section{Introduction}\label{sec:introduction}

A \emph{fake quadric} is a smooth minimal complex surface $S$ of general
type whose numerical invariants agree with those of a smooth quadric in $\mathbb P^3$,
namely $K_S^2=8$ and $p_g(S)=q(S)=0$.
Noether's formula then gives $c_2(S)=4$, so that $S$ has the same Hodge diamond as $\PP^1\times\PP^1$.
All known fake quadrics are uniformized by the bidisk $\HH\times\HH$,
with the exception of those arising from the construction in \cite{RitoExplicitFQ}
for which bidisk uniformization is unknown.
The existence of a fake quadric with a different universal cover remains an open problem
\cite[Problem~5]{CasciniCataneseChenKeum25}.

Bidisk-uniformized fake quadrics fall into two classes, according to whether
the corresponding lattice is reducible or irreducible
\cite{DzambicRoulleau14}.  In the reducible case the surface is isogenous to a
product of curves.  These surfaces may vary in moduli and were classified by
Bauer, Catanese and Grunewald \cite{BauerCataneseGrunewald08}.  In the
irreducible case the lattice is arithmetic and arises from a quaternion
algebra over a totally real number field. The resulting surfaces are
therefore usually called \emph{quaternionic fake quadrics}.
Linowitz, Stover and Voight determined a list of commensurability classes in
which such surfaces may occur \cite{LinowitzStoverVoight19}. 

Quaternionic fake quadrics are, in a sense, analogous to fake projective
planes. Both are rigid surfaces described by arithmetic data, but passing
from such an arithmetic description to explicit projective equations is very
difficult. Indeed, the first explicit equations for a fake projective plane
were obtained by the first author and Keum \cite{BorisovKeum20}, about ten
years after the arithmetic classification of fake projective planes had been
completed.

In this paper we give the first construction by equations of fake quadrics which are
not isogenous to a product of curves.

Our starting point is a construction by the second author given in the paper
\cite{RitoExplicitFQ}, inspired by a result of D\v{z}ambi\'c and Roulleau on
automorphisms of quaternionic fake quadrics \cite{DzambicRoulleau14}: if such
a surface admits an action of $G\cong\ZZ/2\times\ZZ/4$, then its quotient is a
numerical Godeaux surface with two $\mathsf A_1$ and two $\mathsf A_3$ singularities.
In that paper \cite{RitoExplicitFQ}, such a singular
$\ZZ/2$-Godeaux surface $X$ was constructed, and the
divisibility relations needed for the converse construction were established.
The computations in fact yielded two solutions,
corresponding to two non-isomorphic Godeaux surfaces with the same singular
set. Only one of the solutions was pursued there, but here we consider both.
Pardini's theory then yields smooth $G$-coverings $S\to X$, where $S$ is
a fake quadric.  The classification of product-quotient Godeaux surfaces,
together with the lifting theorem proved by Gleissner and Ruhland in the
appendix to \cite{RitoExplicitFQ}, shows that $S$ is not isogenous to a
product.  This approach proves the existence of $S$, but does not provide
its projective model.

Here we make the covering surfaces explicit as projective varieties.  A careful
analysis of the covering data produces exactly eight fake quadrics
$S_1,\ldots,S_8$.  For each $i$, we compute the homogeneous ideal of the
bicanonical model of $S_i$ in $\PP^8$.  These equations make the geometry of
the surfaces accessible and allow us to prove that all eight are rigid.
Moreover, the non-product argument of \cite{RitoExplicitFQ} applies to every
$S_i$, since each carries a $G$-action whose quotient is one of the two
Godeaux surfaces.  Thus the $S_i$ give the first explicit projective models of
fake quadrics which are not isogenous to a product of curves.
The question of whether these surfaces are uniformized by the bidisk remains open.

The paper is organized as follows. In Section \ref{sec:godeaux-construction} we recall
the two singular Godeaux surfaces, the abelian-cover construction, and the argument 
showing that the resulting fake quadrics are not isogenous to a product. 
In Subsection \ref{subsec:explicit-equations} we we describe how to compute the 
bicanonical equations of the covering surfaces. 
In Subsection \ref{subsec:verification} we prove that the 
resulting projective models are indeed bicanonically embedded fake quadrics. 
In Subsection \ref{subsec:rigidity} we prove their rigidity,
and in Subsection \ref{subsec:distinguishing-surfaces} we show that the eight surfaces
are not pairwise isomorphic.

\subsubsection*{Notation}

As usual, the holomorphic Euler characteristic of a surface \(S\) is
denoted by \(\chi(\mathcal O_S)\), the geometric genus by \(p_g(S)\),
the irregularity by \(q(S)\), and a canonical divisor by \(K_S\).
A $(-m)$-curve is a curve isomorphic to $\mathbb P^1$ with self-intersection $-m$.
Linear equivalence of divisors is denoted by $\equiv$.

\subsubsection*{Acknowledgments}

The second author was financed by Portuguese Funds through FCT
(Funda\c{c}\~ao para a Ci\^encia e a Tecnologia) within the Project
UID/00013/2025: Centro de Matem\'atica da Universidade do Minho (CMAT/UM).

\section{The Godeaux surface and the abelian-cover construction}
\label{sec:godeaux-construction}

In this section we recall the ingredients of \cite{RitoExplicitFQ} that will
be used throughout the paper.

\subsection{The Godeaux surface and its covering}\label{subsec:singular-godeaux}

The construction starts with a normal $\ZZ/2$-Godeaux surface $X$.  It has
only rational double points and
\[
 K_X^2=1,\qquad p_g(X)=q(X)=0,\qquad
 \Sing(X)=2\mathsf A_1+2\mathsf A_3.
\]
It was found inside the explicit seven-dimensional family $\mathcal M_2^1$ of
$\ZZ/2$-Godeaux surfaces given in \cite{DiasRito}.  Finite-field
experiments first revealed an unexpected four-dimensional locus of surfaces
with four nodes.  Further interpolation produced a two-dimensional family
with six nodes and, inside it, a one-dimensional locus with singularities
$4\mathsf A_1+\mathsf A_3$.  Special points on the latter locus have singular set
$2\mathsf A_1+2\mathsf A_3$.  The defining data were then lifted to characteristic zero.

Let
\[
                         \rho\colon X'\longrightarrow X
\]
be the minimal resolution.
We label the exceptional $(-2)$-curves $N_i$ so that $N_1$
and $N_2$ lie over the two nodes, whereas
$N_3\!\mathbin{-}\!N_4\!\mathbin{-}\!N_5$ and $N_6\!\mathbin{-}\!N_7\!\mathbin{-}\!N_8$
are the two \(\mathsf A_3\)-chains.

Two irreducible smooth curves $C,D\subset X$ were computed in
\cite{RitoExplicitFQ}.  Let $C',D'\subset X'$ be their strict transforms.
They satisfy
\begin{align}
  8K_{X'} \;&\equiv\;
     4C'+2N_1+N_3+2N_4+3N_5+3N_6+2N_7+N_8,
       \label{eq:four-divisibility}\\
  4K_{X'} \;&\equiv\;
     2D'+N_1+N_2+N_6+2N_7+N_8.
       \label{eq:two-divisibility}
\end{align}
These are the $4$- and $2$-divisibility relations from which the cover is
built.

Set $G=\ZZ/2\times\ZZ/4$.
We briefly describe the reduced building data for a normal $G$-cover in the
form used in \cite{Pardini91}.  The group $G$ has three cyclic subgroups
$H_1,H_2,H_3$ of order $2$ and two cyclic subgroups $H_4,H_5$ of order $4$.
After choosing generators $\psi_j$ of $H_j^*$, the nonzero reduced branch
divisors on $X'$ are
\begin{center}
\begin{tabular}{cc}
\toprule
inertia data & branch divisor\\
\midrule
$(H_1,\psi_1)$       & $N_4+N_7$\\
$(H_2,\psi_2)$       & $N_1$\\
$(H_3,\psi_3)$       & $N_2$\\
$(H_4,\psi_4)$       & $N_3$\\
$(H_4,\psi_4^{-1})$  & $N_5$\\
$(H_5,\psi_5)$       & $N_6$\\
$(H_5,\psi_5^{-1})$  & $N_8$\\
\bottomrule
\end{tabular}
\end{center}
Their total support is a simple normal crossings divisor.

Choose generators $\chi_2,\chi_5$ of $G^*$ of orders $2$ and $4$.  With the
above branch data, Pardini's reduced covering relations are
\begin{align}
  2L_2 &\equiv N_1+N_2+N_6+N_8,
       \label{eq:pardini-two}\\
  4L_5 &\equiv
       2N_1+N_3+2N_4+3N_5+3N_6+2N_7+N_8.
       \label{eq:pardini-four}
\end{align}
Equations \eqref{eq:four-divisibility} and
\eqref{eq:two-divisibility} give the required line bundles explicitly:
\begin{equation}
        L_2\equiv 2K_{X'}-D'-N_7,
        \qquad
        L_5\equiv 2K_{X'}-C'.
        \label{eq:generating-eigensheaves}
\end{equation}
Thus Pardini's existence theorem produces a smooth $G$-cover
\[
                         \pi'\colon S'\longrightarrow X'.
\]
Let $S$ be the surface obtained by contracting the curves in $S'$ lying over
the exceptional divisor of $\rho$.  It is shown in \cite{RitoExplicitFQ} that
$S$ is a smooth minimal surface of general type with
$K_S^2=8$ and $p_g(S)=q(S)=0$. Thus $S$ is a fake quadric.

\subsection{Why the cover is not of product type}
\label{subsec:not-product}

We recall the argument from \cite{RitoExplicitFQ} that $S$ is not a product-quotient surface.
It will apply without change to each of the eight covers constructed in this
paper.

Suppose, for a contradiction, that \(S\) is a product-quotient surface,
say
\[
        S\cong (C_1\times C_2)/\Gamma,
\]
where \(g(C_1),g(C_2)\geq 2\). Since \(S\) is a fake quadric, the
product-quotient formula \cite[Corollary~1.6]{BauerPignatelli12} forces
the basket correction term to vanish. Hence the action of \(\Gamma\) is
free, and \(S\) is isogenous to a product.

The surface \(S\) carries the action of
\(G=\ZZ/2\times\ZZ/4\) coming from the cover. The lifting theorem of
Gleissner and Ruhland \cite[Appendix]{RitoExplicitFQ} shows that every
automorphism of a variety isogenous to a product lifts to the covering
product. The \(G\)-action therefore lifts to \(C_1\times C_2\), and
there is a finite group
\[
        \widetilde\Gamma<\Aut(C_1\times C_2)
\]
such that
\[
        X=S/G\cong(C_1\times C_2)/\widetilde\Gamma.
\]
Thus \(X\) is a product-quotient Godeaux surface. If the action of
\(\widetilde\Gamma\) is unmixed, the classification in
\cite{BauerPignatelli12} excludes the basket
\(2\mathsf A_1+2\mathsf A_3\). If it is mixed, the classification in
\cite{FrapportiPignatelli15} permits this basket only when the Godeaux
surface has fundamental group \(\ZZ/4\). Since
\(\pi_1(X)\cong\ZZ/2\), this is a contradiction. Thus \(S\) is not
isogenous to a product of curves.

\section{Explicit projective models}
\label{sec:explicit-models}

The construction in Section~\ref{sec:godeaux-construction} recalled the
existence of the covering surfaces and the non-product argument.
We now compute their bicanonical models, distinguish the resulting
fake quadrics, and verify their rigidity.

\subsection{Finding explicit equations of the covering surfaces}
\label{subsec:explicit-equations}

\noindent{\bf Step 1.}\newline
Let \(Y\longrightarrow X\) be the universal covering of one of the two singular $\mathbb Z/2$-Godeaux surfaces $X$,
with covering involution \(\iota\), and let \(Z\longrightarrow Y\) be the \(G\)-cover
obtained by pulling back the construction of
Subsection~\ref{subsec:singular-godeaux}, where
\[
 G=\langle a\rangle\times\langle b\rangle
   \cong \mathbb Z/2\times\mathbb Z/4,
 \qquad |a|=2,\quad |b|=4.
\]
Assume there is an involution \(\tau\) of \(Z\) lifting \(\iota\) and commuting
with \(G\).  We call a section \emph{even} or \emph{odd} according as its
\(\tau\)-eigenvalue is \(+1\) or \(-1\).

We retain the generators \(\chi _2,\chi _5\) of \(G^*\) introduced in
Subsection~\ref{subsec:singular-godeaux}, of orders $2$ and $4$,
respectively.

Set
\[
 H=\langle b^2\rangle,\qquad
 \overline G=G/H\cong(\mathbb Z/2)^2,\qquad
 Y'=Z/H.
\]
We thus have a tower
\[
        Z\xrightarrow{\ r\ }Y'
        \xrightarrow{\ q\ }Y
        \longrightarrow X.
\]
All maps in this tower are unramified in codimension one, so canonical
divisors pull back along them.  The characters of \(G\) which are trivial
on \(H\) are precisely those whose order divides $2$, hence
\[
 \overline G^*=G^*[2]
   =\{1,\chi _2,\chi _5^2,\chi _2\chi _5^2\}.
\]

Fix a non-trivial character \(\chi\in\overline G^*\).  Let
\[
 u_{\chi,+},u_{\chi,-}\in H^0(Y',2K_{Y'})
\]
be, respectively, even and odd \(\chi\)-eigensections.  Since
\(\chi^2=1\), their quadratic products are \(\overline G\)-invariant and
therefore descend uniquely to sections
\(F_{\chi,1},F_{\chi,2},F_{\chi,3}\in H^0(Y,4K_Y)\), characterized by
\[
 q^*F_{\chi,1}=u_{\chi,+}^2,\qquad
 q^*F_{\chi,2}=u_{\chi,-}^2,\qquad
 q^*F_{\chi,3}=u_{\chi,+}u_{\chi,-}.
\]
The first two sections are even under \(\iota\), the third is odd, and
\begin{equation}
        F_{\chi,1}F_{\chi,2}=F_{\chi,3}^2.
        \label{eq:rank-one-relation}
\end{equation}

We write \(F_{\chi,1},F_{\chi,2}\) in a basis of the even part of
\(H^0(Y,4K_Y)\), and \(F_{\chi,3}\) in a basis of its odd part, obtaining a
system of polynomial equations in their coefficients.
The Mathematica \cite{Mathematica} computation gives three solutions,
corresponding to the three non-trivial order $2$ characters
$\chi _2, \chi _5^2, \chi _2\chi _5^2$.

Each solution determines an even and an odd bicanonical eigensection on
\(Y'\).  We therefore obtain $6$ new sections of \(H^0(Y',2K_{Y'})\),
one even and one odd for each non-trivial character of \(\overline G\).
Together with the pullback of the $4$-dimensional space
\(H^0(Y,2K_Y)\), whose even and odd parts both have dimension $2$, these
$10$ sections form a basis of
\(H^0(Y',2K_{Y'})\).

To compute the relations among the sections in \(H^0(Y',2K_{Y'})\), we proceed as follows.
For a generic choice of a closed point \(y\) of the cone over \(Y\), the three triples
\((F_{\chi,1},F_{\chi,2},F_{\chi,3})\) determine each pair
\((u_{\chi,+},u_{\chi,-})\) up to a common sign, and hence give
\(2^3=8\) choices for the six new coordinates.  Replacing \(y\) by
\(-y\) and simultaneously changing the signs of all six new coordinates
does not change the corresponding point of \(\mathbb P^9\).  The eight
choices therefore give four points of \(Y'\) over a general point of \(Y\).
Repeating this construction for sufficiently many points gives a large
set of points \(P_1,\ldots,P_N\) on the bicanonical model
\(Y'\subset\mathbb P^9\). For each degree \(d\), a general homogeneous
polynomial of degree \(d\) in the ten coordinates is evaluated at these
points. Requiring it to vanish at every \(P_i\) gives a linear system in
its coefficients. The kernel of the corresponding evaluation matrix is
the space of degree-\(d\) relations. Thus, once sufficiently many points
of \(Y'\) have been computed, finding its defining relations reduces to linear algebra.

We note that one of the sections \(F_{\chi,i}\in H^0(Y,4K_Y)\) has divisor \(2D\), where
$D$ is the curve described in Subsection \ref{subsec:singular-godeaux}.\\

\noindent{\bf Step 2.}\newline
It remains to recover the double cover \(r\colon Z\to Y'\).  There is a
unique non-trivial character of \(\overline G\) which is the square of an
order $4$ character of \(G\), namely
\[
        \tilde\chi=\chi _5^2.
\]
Let \(s\in H^0(Y',2K_{Y'})_{\tilde\chi}\) be one of the two (even or odd) bicanonical
sections of this character.  If \(\psi\in G^*\) has order $4$ and
\(\psi^2=\tilde\chi\), let
\[
        v_{\psi,+},v_{\psi,-}\in H^0(Z,2K_Z)
\]
be even and odd \(\psi\)-eigensections.  Their squares and product are
invariant under \(H\), and hence descend to sections
\[
        g_1,g_2,g_3\in H^0(Y',4K_{Y'})_{\tilde\chi}
\]
satisfying
\[
 r^*g_1=v_{\psi,+}^2,\qquad
 r^*g_2=v_{\psi,-}^2,\qquad
 r^*g_3=v_{\psi,+}v_{\psi,-}.
\]

Since both \(s\) and \(g_i\) transform according to the order $2$
character \(\tilde\chi\), their product is \(\overline G\)-invariant.
Thus
\[
        f_i=s g_i\in H^0(Y',6K_{Y'})^{\overline G},
        \qquad i=1,2,3.
\]
all vanish on \(s=0\), and they satisfy
\begin{equation}
        f_1f_2=f_3^2.
        \label{eq:second-rank-one-relation}
\end{equation}
If \(s\) is even, then \(f_1,f_2\) are even and \(f_3\) is odd. If
\(s\) is odd, the parities are reversed.

Conversely, a solution of \eqref{eq:second-rank-one-relation} in the
\(\overline G\)-invariant subspace, with each \(f_i\) divisible by \(s\),
gives sections \(g_i=f_i/s\) satisfying \(g_1g_2=g_3^2\).
Among the six sections obtained in the first step,
exactly one---necessarily one of the two
\(\tilde\chi\)-eigensections---produces solutions.
The Mathematica \cite{Mathematica} computation yields
eight additional bicanonical sections, one even and one odd in each of the
four order $4$ character spaces
\[
 \chi _5,\qquad \chi _5^{-1},\qquad
 \chi _2\chi _5,\qquad \chi _2\chi _5^{-1}.
\]
Together with the $10$ pullbacks from \(Y'\), they form a basis of
\[
        H^0(Z,2K_Z),\quad \text{with}\quad h^0(Z,2K_Z)=18.
\]

The square roots of these new sections cannot be chosen independently.
We compute several triple products of the sections which are pullbacks of known
sections on \(Y\).  Evaluating these identities at a point relates the
signs of the corresponding square roots and determines the compatible
lifts to \(Z\).  Repeating this construction at sufficiently many points
and interpolating the resulting values gives the homogeneous relations
among the eighteen bicanonical coordinates of \(Z\).

We note that one of those $8$ additional bicanonical sections is the pullback of
the curve $C$ described in Subsection \ref{subsec:singular-godeaux}.\\

\noindent{\bf Step 3.}\newline
For each $g\in\{1,a,b^2,ab^2\}$, set
\[
S_g:=Z/\langle g\tau\rangle.
\]
Since $g\tau$ commutes with $G$ and induces $\iota$ on $Y$, the group
$G$ acts on $S_g$, and
\[
S_g/G\cong X.
\]
Thus the four involutive lifts of $\iota$ give four fake quadrics covering the same
Godeaux surface $X$.
Applying the construction to the two Godeaux surfaces produces eight fake quadrics in total.

In our examples, the surfaces \(X\) and \(Y\) are defined over \(\mathbb Q\).
The solutions of the first step can be scaled so that \(Y'\) is also defined over
\(\mathbb Q\), while \(Z\) can be taken over the Gaussian numbers \(\mathbb Q(i)\).
We were able to tweak the basis so that all fake quadrics are defined over the rational numbers
(although this was no longer an eigenbasis of the action).

\subsection{Verification of the projective models}
\label{subsec:verification}

We now verify directly from the equations that the eight projective models obtained above are smooth 
fake quadrics and that the given embeddings are bicanonical. All computations in Subsections \ref{subsec:verification},
\ref{subsec:rigidity}, \ref{subsec:distinguishing-surfaces} were carried out in Magma \cite{BCP}.
The corresponding code and data files are available as ancillary files accompanying this paper.

\begin{theorem}
Each of the eight projective models computed above is a fake quadric embedded by the bicanonical system.
\end{theorem}

\begin{remark}
Moreover, by construction each of these eight surfaces has an action of the group
$\mathbb Z/2\times\mathbb Z/4$ such that the corresponding quotient is a $\mathbb Z/2$-Godeaux surface
with singular set $2\mathsf A_1+2\mathsf A_3$.
These surfaces are not isogenous to a product of curves by the results in \cite{RitoExplicitFQ}.
\end{remark}

\begin{proof}
Each surface $S$ is defined over $\mathbb Q$, hence also over $\mathbb Z$. Let
$\overline S$ be its reduction modulo $p=97$.
Magma verifies that the Hilbert polynomials of $S$ and $\overline S$ coincide,
and hence that this family is flat.  We shall use the upper semicontinuity theorem:
the dimension of a cohomology group may jump up on reduction modulo
$p$, but cannot jump down.

Magma first verifies over $\mathbb Q$ that indeed $\dim S=2$.
It also computes $h^0(\overline S,\mathcal O_{\overline S})=1$.
By semicontinuity, $h^0(S,\mathcal O_S)=1.$ This implies that $S$ is connected.

We next check smoothness.  Let $I_S$ be the homogeneous ideal of $S$ and
let $J_S$ be the Jacobian matrix of a set of generators.  Since
$\dim S=2$ in $\mathbb P^8$, the relevant minors have size $6\times6$.
We choose $500$ random such minors, reduce them modulo $97$, and adjoin
them to the defining equations of $\overline S$.  For each of the eight
models, Magma verifies that the projective scheme over $\mathbb F_{97}$
defined by these equations is empty.

Let $I'_S\subset\mathbb Q[x_0,\ldots,x_8]$ be the homogeneous ideal
generated by $I_S$ and the chosen minors, and let $\overline{I'_S}$ denote
its reduction modulo $97$.  Since the projective scheme defined by
$\overline{I'_S}$ is empty, the Hilbert polynomial of
$\mathbb F_{97}[x_0,\ldots,x_8]/\overline{I'_S}$ is zero.  Hence
\[
 (\overline{I'_S})_d=\mathbb F_{97}[x_0,\ldots,x_8]_d
\]
for all sufficiently large $d$.  For any such $d$, the degree-$d$ part of
$\overline{I'_S}$ is spanned by the products of its generators with
monomials of the appropriate degrees.  Thus the preceding equality is
equivalent to a certain coefficient matrix having full rank.  After
clearing denominators, this matrix is the reduction modulo $97$ of the
corresponding matrix over $\mathbb Q$.  Since it has full rank modulo $97$,
it also has full rank over $\mathbb Q$: a maximal minor which is non-zero
modulo $97$ is non-zero over $\mathbb Q$.  Consequently,
\[
 (I'_S)_d=\mathbb Q[x_0,\ldots,x_8]_d.
\]
Thus the projective scheme defined by $I'_S$ is empty over $\mathbb Q$.

It follows that at every point of $S$ at least one of the chosen minors is
non-zero.  Thus the Jacobian has rank at least $6$, and hence the Zariski
tangent space has dimension at most $2$.  Every point on a two-dimensional
component is therefore regular.  Since $S$ is connected and has dimension
$2$, there can be no additional component: a disjoint component would
contradict connectedness, whereas a component meeting a two-dimensional
component would give a non-regular point.  Consequently $S$ is smooth and
irreducible, hence is a projective surface.

Magma now gives
$$
        p_g(\overline S)=q(\overline S)=0,
\qquad {\rm i.e.}\qquad 
h^1(\overline S,\mathcal O_{\overline S})=h^2(\overline S,\mathcal O_{\overline S})=0.
$$
Semicontinuity then yields
\[
        p_g(S)=q(S)=0,
        \qquad \chi(\mathcal O_S)=1.
\]

It remains to prove that the given embedding of $S$ is its bicanonical
embedding.  Let $D$ be a hyperplane section of $S$, and let
$\overline D$ be its reduction to $\overline S$.  For each model, Magma computes the divisor $2K_{\overline S}$ modulo
$97$ and verifies that
\[
        \overline D\equiv2K_{\overline S}.
\]
Since $q(\overline S)=p_g(\overline S)=0$, it follows that
\[
\begin{aligned}
 h^1\bigl(\overline S,
      \mathcal O_{\overline S}(2K_{\overline S}-\overline D)\bigr)
 &=h^1\bigl(\overline S,
      \mathcal O_{\overline S}(\overline D-2K_{\overline S})\bigr)
 =h^1(\overline S,\mathcal O_{\overline S})
 =q(\overline S)=0,\\
 h^2\bigl(\overline S,
      \mathcal O_{\overline S}(2K_{\overline S}-\overline D)\bigr)
 &=h^2\bigl(\overline S,
      \mathcal O_{\overline S}(\overline D-2K_{\overline S})\bigr)
 =h^2(\overline S,\mathcal O_{\overline S})
 =h^0(\overline S,K_{\overline S})
 =p_g(\overline S)=0.
\end{aligned}
\]
By upper semicontinuity, the corresponding cohomology groups over
$\mathbb Q$ also vanish:
\begin{equation}
\begin{aligned}
 h^1\bigl(S,\mathcal O_S(2K_S-D)\bigr)
 &=h^1\bigl(S,\mathcal O_S(D-2K_S)\bigr)=0,\\
 h^2\bigl(S,\mathcal O_S(2K_S-D)\bigr)
 &=h^2\bigl(S,\mathcal O_S(D-2K_S)\bigr)=0.
\end{aligned}
\label{eq:inverse-line-bundle-vanishings}
\end{equation}

The Hilbert polynomial of the given projective model is
\(
        16m^2-8m+1.
\)
Comparing this with Riemann--Roch for $mD$ gives
$D^2=32,$ $D\mathbin{\cdot}K_S=16$.
Using \eqref{eq:inverse-line-bundle-vanishings} and Riemann--Roch on $S$,
we obtain
\[
\begin{aligned}
 h^0\bigl(S,\mathcal O_S(2K_S-D)\bigr)
   &=\chi\bigl(\mathcal O_S(2K_S-D)\bigr)
     =K_S^2-7,\\
 h^0\bigl(S,\mathcal O_S(D-2K_S)\bigr)
   &=\chi\bigl(\mathcal O_S(D-2K_S)\bigr)
     =3K_S^2-23.
\end{aligned}
\]

If $K_S^2>8$, the two dimensions displayed above are both greater than
$1$, which is impossible.  If $K_S^2<8$, then the second displayed value is negative, which is also impossible.
Consequently
\(
        K_S^2=8
\).
The two dimensions are then both equal to $1$, thus
\[
        D\equiv2K_S
\]
over $\mathbb Q$.

Since $D$ is ample, the canonical divisor $K_S$ is ample.  Thus $S$ is a
minimal surface of general type with
\[
        K_S^2=8,
        \qquad
        p_g(S)=q(S)=0,
\]
and hence is a fake quadric.  Moreover, Kodaira vanishing and
Riemann--Roch give
\[
        h^0(S,2K_S)=\chi(\mathcal O_S)+K_S^2=9.
\]
The nine sections defining the given embedding therefore form the complete
bicanonical system, so the model in $\mathbb P^8$ is the bicanonical model of
$S$.
\end{proof}

\subsection{Rigidity}
\label{subsec:rigidity}

We next prove the rigidity of the eight surfaces by computing cohomology of the tangent bundle.

\begin{theorem}
The eight fake quadrics constructed above are rigid surfaces.
\end{theorem}

\begin{proof}
Let $T_S$ be the tangent bundle of one of the eight fake quadrics $S$.  To prove that \(S\) is rigid,
it suffices to show that
\(
        H^1(S,T_S)=0.
\)
Riemann--Roch formula gives
\[
 \chi(S,T_S)=2K_S^2-10\chi(\mathcal O_S)=16-10=6.
\]
Since a surface of general type has no non-zero holomorphic vector fields,
\(h^0(S,T_S)=0\), and therefore
\[
        h^1(S,T_S)=h^2(S,T_S)-6.
\]
By Serre duality and the rank-two identity
\(\Omega_S^1\cong T_S\otimes K_S\),
\[
\begin{aligned}
 h^2(S,T_S)
   &=h^0(S,\Omega_S^1\otimes K_S)\\
   &=h^0(S,T_S\otimes2K_S)
    =h^0(S,T_S(1)),
\end{aligned}
\]
where \(\mathcal O_S(1)=\mathcal O_S(2K_S)\) for the bicanonical model.

For each of the eight models, we compute the corresponding quantity
after reduction modulo $97$.  More precisely, Magma gives
\(h^0\bigl(\overline S,T_{\overline S}(1)\bigr)=6\).
By upper semicontinuity,
\[
        h^0(S,T_S(1))
        \leq
        h^0\bigl(\overline S,T_{\overline S}(1)\bigr)
        =6.
\]
Thus $h^2(S,T_S)\leq6$, and the preceding Euler-characteristic identity
forces
\[
        h^2(S,T_S)=6,
        \qquad
        h^1(S,T_S)=0.
\]
Hence all eight fake quadrics are rigid.
\end{proof}

\subsection{Distinguishing the surfaces}
\label{subsec:distinguishing-surfaces}

Finally, we distinguish the eight surfaces by considering the pencil of quadrics contained in their bicanonical ideals.

\begin{theorem}
The eight fake quadrics given above are pairwise non-isomorphic.
\end{theorem}

\begin{proof}
For \(j=1,\ldots,8\), put
\[
        V_j=H^0(S_j,2K_{S_j}).
\]
By Riemann--Roch and Kodaira vanishing, \(\dim V_j=9\). The group
\(G\cong\mathbb Z/2\times\mathbb Z/4\) acts naturally on \(V_j\). For
every non-trivial element of order two, the periodicity in the
holomorphic Lefschetz formula gives trace $1$ on \(V_j\). An element of
order four fixes two points of \(S_j\). Each fixed point contributes
\[
        \frac{1}{(1-i)(1+i)}=\frac12,
\]
so its trace on \(V_j\) is again one. Hence the character of \(V_j\) is
the sum of the trivial and the regular characters of \(G\). In particular,
\(V_j\) has an eigenbasis with
\(\mathbb Z/2\times\mathbb Z/4\)-grading
\begin{equation}
 (0,0)\oplus(0,0)\oplus(0,1)\oplus(0,2)\oplus(0,3)
 \oplus(1,0)\oplus(1,1)\oplus(1,2)\oplus(1,3)
 \label{eq:bicanonical-grading}
\end{equation}
which we have also computed explicitly as the consequence of our construction of the covering surface.

Similarly, \(h^0(S_j,4K_{S_j})=49\), and the trace of every non-trivial
element of \(G\) on \(H^0(S_j,4K_{S_j})\) is one. Thus the invariant
eigenspace has dimension seven and each of the other seven eigenspaces
has dimension six. On the other hand,
\eqref{eq:bicanonical-grading} gives
\[
 \dim\!\left(\operatorname{Sym}^2V_j\right)_{(0,0)}=8,
 \qquad
 \dim\!\left(\operatorname{Sym}^2V_j\right)_{(0,2)}=7.
\]
Consequently, the kernel
\[
 I_2(S_j)=\ker\!\left(
   \operatorname{Sym}^2H^0(S_j,2K_{S_j})
   \longrightarrow H^0(S_j,4K_{S_j})
 \right)
\]
contains at least one quadric of character \((0,0)\) and one of
character \((0,2)\). For each \(j=1,\ldots,8\), the exact computation
gives
\(
        \dim I_2(S_j)=2.
\)

Choose a basis \(Q_{j,0},Q_{j,1}\) of \(I_2(S_j)\), and let
\(A_{j,0},A_{j,1}\) be the corresponding symmetric matrices. The
singular members of the pencil are parametrized by the vanishing of
\[
        \Delta_j(s,t)
        :=\det\bigl(sA_{j,0}+tA_{j,1}\bigr).
\]
If \(\Delta_j\) is non-zero, it is a binary form of degree \(9\), and
its zero divisor is the degeneracy divisor of the pencil. For one of
the eight surfaces, \(\Delta_j\) vanishes identically. In this case
every quadric in the pencil is singular and the degeneracy locus is the
whole of \(\mathbb P^1\).

An isomorphism between two of the surfaces would induce an isomorphism
between their bicanonical spaces, and hence a projective equivalence
between the corresponding degeneracy loci. Thus, for some
\[
        M=
        \begin{pmatrix}
        a&b\\ c&d
        \end{pmatrix}
        \in\operatorname{GL}_2(\overline{\mathbb Q}),
\]
the determinant forms corresponding to \(S_j\) and \(S_k\) would
satisfy
\[
        \Delta_j(as+bt,cs+dt)=\Delta_k(s,t),
\]
together with
\[
        1+n(ad-bc)=0
        \quad\text{for some }n,
\]
which imposes the condition \(\det(M)\neq0\).

For each of the \(28\) pairs \(1\leq j<k\leq8\), we formed the
corresponding scheme in the five-dimensional affine space with
coordinates \(a,b,c,d,n\). Exact computations in Magma show that all
these schemes are empty. Hence the eight degeneracy loci are pairwise
non-equivalent over \(\overline{\mathbb Q}\), and consequently the eight
fake quadrics are pairwise non-isomorphic.
\end{proof}

{\small
\bibliographystyle{amsplain}
\bibliography{fake_quadrics_references}
}

\vspace{1cm}

\noindent Lev Borisov \vspace{0.1cm} 
\\Department of Mathematics, Rutgers University
\\Piscataway, NJ 08854
\\ \verb|borisov@math.rutgers.edu| 

\vspace{1cm}

\noindent Carlos Rito
\vspace{0.1cm}
\\ Centro de Matem\'atica, Universidade do Minho - Polo CMAT-UTAD
\vspace{0.1cm}
\\ Universidade de Tr\'as-os-Montes e Alto Douro, UTAD
\\ Quinta de Prados
\\ 5000-801 Vila Real, Portugal
\vspace{0.1cm}
\\ www.utad.pt, {\tt crito@utad.pt}

\end{document}